\documentclass[12pt]{amsart}
\usepackage{amsmath,amsthm,amsfonts,amssymb,amscd,amstext}

\usepackage{graphicx}
\usepackage[utf8]{inputenc}
\usepackage{psfrag}
\usepackage{euscript}
\usepackage{a4wide}
\usepackage{srcltx}
\usepackage{psfrag}
\usepackage{pxfonts}
\usepackage{pslatex}
\numberwithin{equation}{section}

\usepackage{srcltx}
\theoremstyle{plain}

\newtheorem{theorem}{Theorem}[section]
\newtheorem{proposition}[theorem]{Proposition}
\newtheorem{corollary}[theorem]{Corollary}
\newtheorem{lemma}[theorem]{Lemma}

\theoremstyle{definition}
\newtheorem{remark}[theorem]{Remark}
\newtheorem{definition}{Definition}

\newcommand{\RR}{{\mathbb R}}

\newcommand{\SU}{{\mathcal U}}

\newcommand{\la}{\lambda}

\renewcommand{\epsilon}{\varepsilon}

\newcommand{\indi}{\operatorname{ind}}

\newcommand{\per}{\operatorname{Per}}

\newcommand{\cri}{\operatorname{Crit}}

\newcommand{\supp}{\operatorname{supp}}

\newcommand{\sing}{\mathrm{Sing}}

\newcommand{\cO}{\EuScript{O}}

\newcommand{\U}{\EuScript{U}}

\newcommand{\cC}{\EuScript{C}}

\newcommand{\J}{\EuScript{J}}

\newcommand{\ud}{\mathrm{d}}

\begin{document}

\title{Star Flows: a characterization via Lyapunov functions}

 \thanks{This is the last version of one presented at AIMS Conference on Dynamical Systems, Differential Equations and Applications
 2016. A first version of this work was developed during a postdoctoral research of L.S. at IMPA, in 2012. Since then, She has been partially supported by a FAPERJ-Funda\c c\~ao Carlos Chagas Filho de Amparo \`a Pesquisa do Estado do Rio de Janeiro Projects APQ1-E-26/211.690/2021 SEI-260003/015270/2021 and JCNE-E-26/200.271/2023 SEI-260003/000640/2023, by Coordena\c c\~ao de Aperfei\c coamento de Pessoal de N\'ivel Superior CAPES — Finance Code 001, CNPq-Conselho Nacional de Desenvolvimento Cient\'ifico e Tecnol\'ogico PDJ-152926/2016-0 and Projeto Universal 404943/2023-3, Fapesb-Funda\c c\~ao de Amparo \`a Pesquisa do Estado da Bahia Project JCB0053/2013, PRODOC/UFBA 2014 and INCTMat-CAPES postdoctoral fellowship 2012.}
%\thanks{}

%\subjclass{Primary: MSC 37D30 and MSC 37D25; Secondary: MSC 37L45}

%\keywords{Star flows, Singular hyperbolicity, Lyapunov functions, Dominated splitting, Strong homogeneity.}

\author{Luciana Salgado}
\address[L.S.]{ Universidade Federal do Rio de Janeiro, Instituto de Matem\'atica \\
Avenida Athos da Silveira Ramos 149 Centro de Tecnologia - Bloco C,
Cidade Universit\'aria - Ilha do Fund\~ao.\\
Caixa Postal (P.O. Box) 68530 CEP (Zip Code) 21941-909 Rio de Janeiro - RJ - Brazil \\
              Tel.: +55-21-39387909\\}
              \email{lsalgado@im.ufrj.br; lucianasalgado@ufrj.br}



\begin{abstract}
In this work, it is presented a characterization of star property for a $C^1$ vector field based on Lyapunov functions.
  It is also obtained conditions to strong homogeneity for singular sets by using the notion of infinitesimal Lyapunov functions.
  As an application, we obtain some results related to singular hyperbolic sets for flows.

\end{abstract}

\maketitle

\section{Introduction}
\label{intro}

Star systems has been studied by many renowned researchers, among them R. Ma\~n\'e and S. Liao, whom many years ago used it in order to prove the famous \emph{stability conjecture} from Palis and Smale. For more details about star systems, see for instance \cite{ArbMo2013},\cite{shi-gan-wen2014},\cite{GanWen2006},\cite{Liao1980},\cite{Man82},\cite{Man88},\cite{Palis87}.

   \begin{definition}\label{def:star-flow}
A $C^1$ vector field $X$ (or its flow $X_t$) is said to be \emph{star} if it cannot be $C^1$-approximated by ones exhibiting nonhyperbolic periodic orbits/singularities.
\end{definition}

The definition for diffeomorphisms is analogous.

In the case of diffeomorphisms, $C^1$ $\Omega$-structural stability is equivalent to star condition (and equivalent to Axiom A plus no cycle), see e.g. \cite{Man88}.

However, in the case of flows, the situation is much more complex. Although in absence of singularities, Gan and Wen \cite{GanWen2006} proved that nonsingular star flows satisfy Axiom A and no cycle condition, in the singular setting, this is no longer true.

One of the most emblematic example that a star flow does not satisfy Axiom A is the geometric Lorenz Attractor \cite{Guck76}, \cite{Lo63}.

From this, remained the question about singular star vector fields.
\begin{figure}[htpb]
\begin{center}
\includegraphics[scale=0.5]{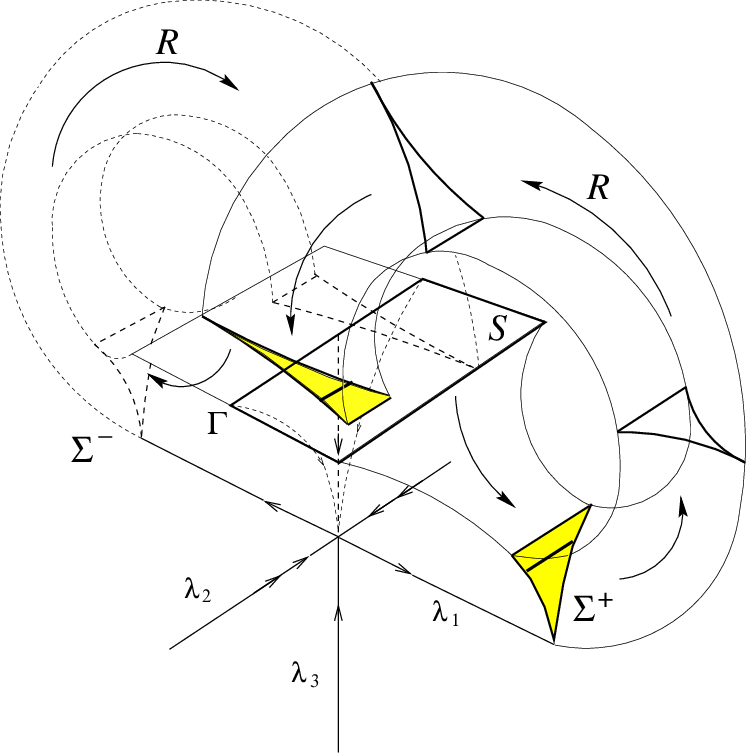}
\caption{Geometric Lorenz attractor is an example of singular star flow}
\label{Lorenz}
\end{center}
\end{figure}

In an attempt to study the behaviour of robust singular attractors like Lorenz ones, Morales, Pac\'ifico and Pujals in \cite{MPP99} defined the so called \emph{singular hyperbolic systems}.

In a remarkable work, W. Tucker \cite{Tu99} solved the $14^\circ$ problem of Smale by using computational tools and normal forms. He proved that Lorenz's system indeed supports a singular hyperbolic attractor, which means that it is a numerical example of star system.

\begin{figure}[htpb]
\begin{center}
\includegraphics[scale=0.7]{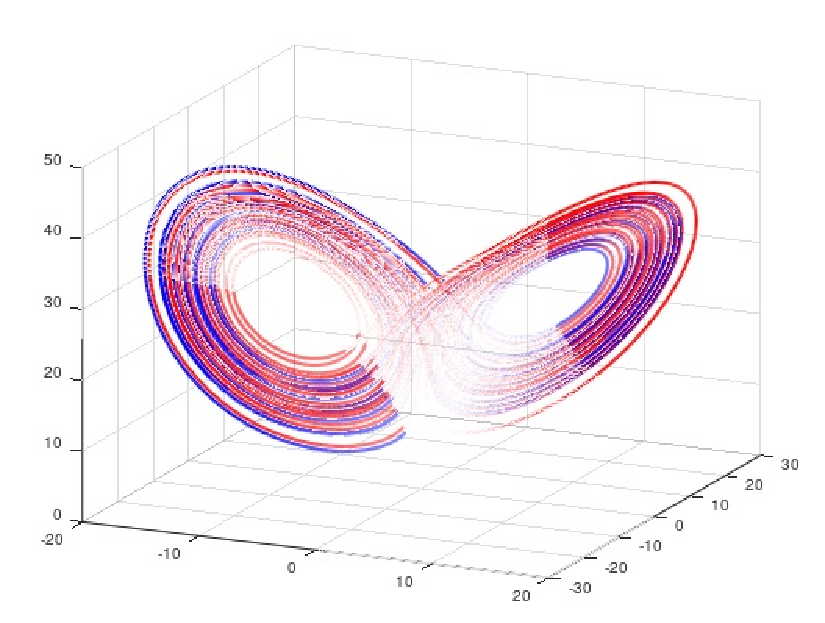}
\caption{Lorenz equations simulated by Octave, it is an example of singular star system}
\label{Lorenzsys}
\end{center}
\end{figure}

Many researchers have worked about this notion in order to understand it as an extension of the hyperbolic theory for invariant sets for flows which are not (uniformly) hyperbolic, but which have some robust properties, certain kind of weaker hyperbolicity and also admit singularities.

The authors in \cite{MPP99} proved that, $C^1$-generically in dimension three, chain recurrence classes are singular hyperbolic. Without the generic assumption this result does not hold, as we can see in \cite{BaMo14}.

In \cite{shi-gan-wen2014}, Gan, Shi and Wen, proved that if a chain recurrent class of a star flow $X$ has homogeneous index for singularities (the same dimension of stable manifold), then this is a singular hyperbolic set of $X$, in any dimension. Moreover, they proved that, $C^1$-generically in dimension four, the chain recurrent set of a star flow is singular hyperbolic.

 In \cite{MPP04}, it has been proved that every robustly transitive singular set for a three dimensional flow is a partially hyperbolic attractor or repeller and the singularities in this set must be Lorenz-like. In \cite{GanLiWen2005}, Gan, Li and Wen extended the result in \cite{MPP04} to higher dimension assuming that the set is also strongly homogeneous. We recall that a vector field $X$ is said to be \emph{strongly homogeneous of index $0 \leq \indi(\Lambda) \leq n-1$} over a set $\Lambda$ whether it cannot be $C^1$-approximated by one which has some hyperbolic periodic orbit of index different of $\indi(\Lambda)$ in a neighborhood $U$ of $\Lambda$. Here, the index $\indi(\cdot)$ means the dimension of the contracting subbundle from the hyperbolic splitting of a hyperbolic periodic orbit .

We recall that a compact invariant set $\Lambda$ is \emph{robustly transitive} for a vector field $X$ if there exist a neighborhood $U$ of $\Lambda$ and a neighborhood $\SU \in \mathfrak{X}^{1}(M)$ of $X$ such that, for every $Y \in \SU$, the maximal invariant set $\Lambda_Y = \cap_{t \in \RR} X_t(U)$ is contained in the interior of $U$ and is non-trivially (not a single orbit) transitive.

Many researchers believed that, at least generically, singular star flows should be singular hyperbolic ones. There are many conjectures about this and various results involving the most varied assumptions, see for instance, \cite{ArbMo2013}, \cite{GanLiWen2005}, \cite{shi-gan-wen2014}.
However, in a recent work, Bonatti and da Luz \cite{BodL17} have defined \emph{multisingular hyperbolicity} which admit singularities of different indices in the same recurrent class, and have shown that this kind of hyperbolicity implies star condition. They also proved that the converse holds on the $C^1$ generic assumption. Moreover, da Luz \cite{dLuz18} announced an example of multisingular hyperbolic flow in dimension 5 which admits, robustly, two singularities of different indices in the same recurrence class. Also, in \cite{CdLYZ}, Crovisier et al simplify this notion by one that does not involve the blow-up of the singular set and the rescaling cocycle of the linear flows. For an early definition involving domination for Linear Poincar\'e flow and hyperbolicity, see \cite{Salg2014}.

In this paper, we prove a full characterization of star condition for flows via $\J$-algebra of Potapov (infinitesimal Lyapunov functions), see \cite{Pota60}, \cite{Pota79}, \cite{Wojtk85}, \cite{Wojtk01}. 

\begin{remark}\label{rmk:comparison}
This characterization encompass all star flows, without generic assumptions. 

In \cite{ArSal2012}, together with V. Ara\'ujo, we used the quadratic forms to characterize some hyperbolic features of flows, such as uniform, partial and $2$-sectional hyperbolicity. 
In a work in progress, one hope to show that this theory can be applied to others kind of hyperbolicities, including above cited cases, such as sectional (in the broad sense of Definition \ref{def:new-def-singhyp}), singular and multi-singular hyperbolicity  as in \cite{BodL17}, \cite{CdLYZ}.  
\end{remark}

We also show a relation between $\J$-algebra and strong homogeneity. Then, we apply this to obtain some results about singular hyperbolic systems.

Recall that a pseudo-euclidean space is a real finite-dimensional vector space endowed with a non-degenerate quadratic form (in case of an euclidean space, a positive definite one).

The $\J$-algebra here means a pseudo-euclidean structure given by a $C^1$ non-degenerate quadratic form $\J$, defined on $\Lambda$, which generates positive and negative cones with maximal dimension $p$ and $q$, respectively, with $p + q = \dim (M)$.

The maximal dimension of a cone in $T_xM$ is the maximal dimension of the subspaces contained in there.

This algebraic/geometric approach has been very useful in the study of weak and uniform hyperbolicity, see \cite{BurnKatok94}, \cite{lewow80}, \cite{lewow89}, \cite{Wojtk85}, \cite{Wojtk01}. In \cite{ArSal2012}, this author jointly with V. Ara\'ujo, obtained characterizations of partial and singular/sectional hyperbolicity based on $\J$-algebra. In \cite{ArSal2015}, the same authors proved an equivalence between dominated splittings for the flow and dominated splittings for the $k$-th exterior powers of the tangent cocycle.

More results relating geometric and algebraic features of singular hyperbolicity can be viewed in \cite{ArSal2012}, \cite{ArSal2015}, \cite{SalgCoelho2017}, for the classical sectional and singular hyperbolicity definitions. Also see \cite{Salg2016} for singular hyperbolicity in a broad sense involving sectional expansion of intermediate dimensions between two and the full dimension of the central subbundle.

Our main result (Theorem \ref{cor:star-flow}) is about the so called \emph{star flows} and its relation with infinitesimal Lyapunov functions.

We recall that a point $p \in M$ is said to be a $C^1$ preperiodic point of $X$ if for any $C^1$ neighborhood $\mathcal{V}$ of $X$ and any neighborhood $U \subset M$ of $p$, there is $g\in \mathcal{V}$ and $q \in U$ such that $q \in \per (g)$. We denote this set by $P_*(X)$, and it is easy to see that it is closed and $X$-invariant.

 We can assume that the periodic points are hyperbolic and we can define a $C^1 i$-preperiodic point $p$ of $X$, $0 \leq i \leq d$, if there are sequences $X_n$ of flows and $p_n$ of periodic points of $X_n$ with index $i$ such that
\begin{align*}
\lim\limits_{n \to \infty} X_n = X \ \textrm{and} \ \lim\limits_{n \to \infty} p_n = p.
\end{align*}
We denote by $P_*^i(X)$ the set of $C^1 i$-preperiodic of $X$ and, then, $P_*(X) = \cup_{i=0}^{d} P_*^i(X)$.

\begin{remark}
The definitions concerning the quadratic forms are given in the next section, precisely in Definition \ref{def:J-separated} and remarks just after this.
\end{remark}

Before state the main result of this paper, let us recall a previous result from \cite{ArSal2012} on the expression of the time derivative of quadratic forms over a cocycle.

\begin{theorem}\cite[Theorem 2.7]{ArSal2012}
\label{thm:QF-deriv}
  Let $X_t$ be a flow defined on a positive invariant subset $U \subset M$, $A_t(x)$ a cocycle over $X_t$ on $U$ and $ D(x)=\lim_{t\to0}(A_t(x)-Id)/t$ its infinitesimal generator. Then,
  \begin{align}
    \J^\prime(v) := \partial_t \J(A_t(x)v) = \langle \tilde J_{X_t(x)}
    A_t(x)v,A_t(x)v\rangle,
    \end{align}
    for all $v\in E_x$ and $x\in U$, where
    \begin{equation}\label{eq:J-separated-tildeJ}
      \tilde J_x:= J\cdot D(x) + D(x)^* \cdot J
    \end{equation}
    and $D(x)^*$ denotes the adjoint of the linear map
    $D(x):E_x\to E_x$ with respect to the adapted inner
    product at $x$.

\end{theorem}

Precisely, here it is proved a full characterization of star property for flows based on Lyapunov quadratic functions. 

\begin{theorem}[Sun's Theorem]\label{cor:star-flow}
A vector field $X \in \mathfrak{X}^{1}(M)$ is star if, and only if, satisfies all of the next properties:
\begin{enumerate}
\item there is a neighborhood $U$ of $P_*(X)$ and a field of quadratic forms $\J$ with index $0 \leq \indi \leq \dim(M) - 1$ defined on the preperiodic set $P_*(X)$, $C^1$ along the flow direction over each preperiodic orbit, such that $X$ is strictly $\J$-separated on every $p \in P_*(X)$;

\item for every $\sigma \in \sing(X\vert_U)$ and $\forall v \in T_\sigma M, \J'(v) > 0$;

\item the linear Poincar\'e flow $P^t$ associated to each preperiodic orbit $\gamma$ of $X\vert_U$ is strictly $\J$-monotone.
\end{enumerate}
\end{theorem}

Let us explain why we are not requiring strict monotonicity in a whole nonsingular neighborhood of $X$. In fact, under certain assumptions, if we have strict monotonicity over any compact invariant nonsingular subset then it is a hyperbolic subset (see \cite[Theorem D]{ArSal2012}). The idea here is to extend the characterization for any star flow, since it is already known the existence of star flows which are not neither uniformly hyperbolic nor singular/sectional hyperbolic (see \cite{BodL17}).

Now, we present the definition of strong homogeneity.

\begin{definition}
\label{def:str-hom-set}
We say that a set $\Lambda$ is strongly homogeneous of index $\indi$ for a flow $X_t$, if there exist neighborhoods $U$ of $\Lambda$ and $\U$ of $X$ such that all periodic orbits in $U$ with respect to any flow in $\U$ have index $\indi$.

\end{definition}

   The second result guarantees that a compact connected invariant set for a flow is strongly homogeneous under the existence of a field of non-degenerate quadratic forms $\J$ defined on a neighborhood of this set such a way that the flow derivative $DX_t$ keeps positive cones $C_+(x) : = \{0\} \cup \{ v \in T_xM ; \J(x) v > 0\}$ invariants, i.e., $DX_t(C_+(x)) \subset C_+(X_t(x))$, for all $t>0$, $x \in \Lambda$ and the projected quadratic forms on the normal bundle, with respect to the flow direction, are monotonic functions.

\begin{theorem}
\label{mthm:strong-homog-equiv}
A compact invariant set $\Lambda$ for $X \in \mathfrak{X}^{1}(M)$ is strongly homogeneous with index $\indi$ if and only if there is a neighborhood $U\subset M$ of $\Lambda$ and a continuous field of non-degenerate quadratic forms $\J$ on $U$ with fixed index $\indi(\J) = \indi$ such that the preperiodic set $P_*(X)$ of $X\vert_{U}$ is strictly $\J$-separated and the associated linear Poincar\'e flow $P^t$ is strictly $\J$-monotone on $P_*(X\vert_{U})$.
\end{theorem}

If $\Lambda$ is the maximal invariant set of a trapping region $U$ and we require, in addition, that the field direction must be inside the non-positive cone, we obtain an equivalence between the existence of such a quadratic forms and singular hyperbolicity, as in \cite[Theorem D]{ArSal2012}.

Note that, in the last result, monotonicity is only required on preperiodic orbits.

If we require it over any nonsingular compact invariant subsets in $\Lambda$, it is possible to evaluate the index of a singularity, once it is accumulated by regular orbits.

\begin{corollary}\label{mcor:strong-homog}
Let $\Lambda$ be a maximal invariant set of a neighborhood $U$ for $X \in \mathfrak{X}^{1}(M)$. Then, $\Lambda$ is strongly homogeneous with index $\indi$ and $\indi(\sigma) \geq \indi$ for all $\sigma \in \sing(X\vert_{\Lambda})$ if there is a field of non-degenerate $C^1$ quadratic forms $\J$ on $U$ with index $\indi(\J) = \indi$ such that $X$ is
strictly $\J$-separated, the associated linear Poincar\'e flow $P^t$ is strictly $\J$-monotone on every compact invariant nonsingular subset $K$ of $U$ and for every $\sigma \in \sing(X\vert_U)$ and $\forall v \in T_\sigma M, \J'(v) > 0$.
\end{corollary}

\vspace{0.1in}

We may ask if the converse of Corollary \ref{mcor:strong-homog} is valid. But, just by supposing strongly homogeneity, we could not obtain the field of quadratic forms, because we need some kind of decomposition on the tangent bundle to create the cones.

\begin{theorem}
\label{mthm:strong-homog-to-quad}
Let $\Lambda$ be a 
compact invariant set whose singularities are hyperbolic (if any) and accumulated by regular orbits for a $C^1$ vector field $X$, which is strongly homogeneous with index $\indi$ and $\indi(\sigma) > \indi$ for all $\sigma \in \sing(X\vert_{\Lambda})$.
Then, there exists a field of non-degenerate quadratic forms $\J$ on $\Lambda$ with index $\indi(\J) = \indi(\Lambda)$ for which $X$ is $\J$-separated and the associated linear Poincar\'e flow $P^t$ is strictly $\J$-monotone on every compact invariant nonsingular subset $\Gamma$ of $\Lambda$.
\end{theorem}

As an application, we obtain some results about partial hyperbolicity for robustly transitive strongly homogeneous singular sets of \cite{GanLiWen2005} and transitive set of \cite{ArbMo2013}.

The text is organized as follow. In the first section, it is given the main definitions and stated the main  results. In second section, it is presented the main tools by using the notion of $\J$-algebra of Potapov. In third section, it is given some applications concerning singular hyperbolicity. In fourth section is proved the Theorems \ref{cor:star-flow}, \ref{mthm:strong-homog-equiv} and Corollary \ref{mcor:strong-homog}. In fifth section is given the proof of Theorem \ref{mthm:strong-homog-to-quad}.

\section{Some definitions and auxiliary results}
\label{sec:prelim-definit}

Now, we give some definitions.

Let $M$ be a connected compact finite $d$-dimensional
manifold, $d \geq 3$, without boundary, together
with a flow $X_t : M \to M, t \in \mathbb{R}$ generated by a
$C^1$ vector field $X: M \to TM$.

An \emph{invariant set} $\Lambda$ for the flow of $X$ is a
subset of $M$ which satisfies $X_t(\Lambda)=\Lambda$ for all
$t\in\mathbb{R}$.

A \emph{trapping region} $U$ for a flow $X_t$ is an
open subset of the manifold $M$ which satisfies: $X_t(U)$ is
contained in $U$ for all $t>0$; and there exists $T>0$ such
that $\overline{X_t(U)} $ is contained in the interior of
$U$ for all $t>T$. The maximal invariant set
$\Lambda_X(U):= \cap_{t \geq 0} X_t(U)$ of $U$ is called
an \emph{attracting set}. An attracting set for $X$ which is transitive
is called an \emph{attractor} for $X$. A \emph{repeller} for $X$ is an attractor
for $-X$.

We say that a set $\Lambda$ is \emph{Lyapunov stable} if for every
neighborhood $U$ of $\Lambda$ there is another one $V \subset U$ such that
every point $p \in V$ has its forward orbit contained in $U$.

A \emph{singularity} for the vector field $X$
is a point $\sigma\in M$ such that $X(\sigma)=0$ or,
equivalently, $X_t(\sigma)=\sigma$ for all $t \in \mathbb{R}$. The
set formed by singularities is the \emph{singular set of $X$}
denoted $\sing(X)$ and $\per(X)$ is the set of periodic points of $X$.
  We say that a \emph{singularity is hyperbolic} if the eigenvalues of the derivative
$DX(\sigma)$ of the vector field at the singularity $\sigma$
have nonzero real part. The set of critical elements of $X$ is
the union of the singularities and the periodic orbits of $X$, and will be denoted by $\cri(X)$.

We recall that an hyperbolic set $\Lambda$ for a flow $X_t$ is an
invariant subset of $M$ with a decomposition $T_\Lambda M= E^s\oplus E^X \oplus E^u$
of the tangent bundle over $\Lambda$ which is a continuous splitting,
where $E^X$ is the direction of the vector field, the
subbundles are invariant under the derivative $DX_t$ of the flow
$$
  DX_t E^{s,X,u}_x=E^{s,X,u}_{X_t(x)},\
  x\in\Lambda, \ t \in \mathbb{R};
$$
where $E^s$ is uniformly contracted by $DX_t$ and $E^u$ is
uniformly expanded: there are $K,\lambda>0$ so that
\begin{align}\label{eq:def-hyperbolic}
\|DX_t\mid_{E^s_x}\|\le K e^{-\lambda t},
  \quad
  \|DX_{-t} \mid_{E^u_x}\|\le K e^{-\lambda t},
  \quad x\in\Lambda, \quad t\in\mathbb{R}.
\end{align}

 We say that a point  $x \in  M$ is {\em non-wandering} for $X$ provided for every neighborhood $U$ of $x$ there is a sequence of times $\tau_n \to \infty$ such that $X_{\tau_n}(x) \in U, \forall n \in \mathbb{N}$.
 We denote by $\Omega(X)$ the non-wandering set of $X$.

 Given $\epsilon, \tau > 0$, an $\epsilon$-chain from $x_0$ to $x_l$ for $X$ is a sequence $\{x_0, x_1, \cdots, x_l, t_1, \cdots, t_l\}$ such that for all $0 \leq j \leq l-1$, $t_j \geq \tau$ and the distance $d(X_{t_j}(x_{j}), x_{j+1}) < \epsilon$.

 We define the {\em chain recurrent set} of $X$ by $R(X) = \{x \in M; \textrm{there \ is \ an} \ \epsilon-\textrm{chain \ from} \ x \ \textrm{to} \ x\}$. We say that two points are {\em chain equivalent} provided, given $\epsilon > 0$, there is an $\epsilon$-chain from $x$ to $y$ and from $y$ to $x$. It is known that this is an equivalence relation, the equivalence classes are called {\em chain components} of $R(X)$ and, for flows, the components are actually the connected components of $R(X)$. If $X$ admits a single chain component on an invariant set $\Lambda$, we say that $X$ is {\em  chain transitive} on $\Lambda$. See \cite{Rob99}, for instance.

From Pugh's $C^1$ closing lemma we have $\Omega(X) \subset P_*(X) \subset R(X)$.

\begin{remark}\label{rmk:wen-preperiod}
As proved by Wen \cite{wen00}, $C^1$ preperiodic sets do not explode under $C^1$ perturbations, i.e.,
for any $0 \leq i \leq d$ and for any neighborhood $U$ of $P_*(X)$, there is a $C^1$ neighborhood $\mathcal{V}$ of $X$ such that $P_*(Y) \subset U$, for any $Y \in \mathcal{V}$.

Indeed, for example in \cite{shub87}, we can see that the recurrent set do not explode under $C^0$ perturbations, i.e., if $U$ is a neighborhood of $R(X)$ there is a neighborhood $\mathcal{V}$ of $X$ such that $R(Y) \subset U$ for all $Y \in \mathcal{V}$.
\end{remark}

Next, we explain the kind of hyperbolicities we are dealing with.

\begin{definition}\label{def1}
  A \emph{dominated splitting} over a compact invariant set $\Lambda$ of $X$
  is a continuous $DX_t$-invariant splitting $T_{\Lambda}M =
  E \oplus F$ with $E_x \neq \{0\}$, $F_x \neq \{0\}$ for
  every $x \in \Lambda$ and such that there are positive
  constants $K, \lambda$ satisfying
  \begin{align}\label{eq:def-dom-split}
    \|DX_t|_{E_x}\|\cdot\|DX_{-t}|_{F_{X_t(x)}}\|<Ke^{-\la
      t}, \ \textrm{for all} \ x \in \Lambda, \ \textrm{and
      all} \,\,t> 0.
  \end{align}
\end{definition}

A compact invariant set $\Lambda$ is said to be
\emph{partially hyperbolic} if it exhibits a dominated
splitting $T_{\Lambda}M = E \oplus F$ such that subbundle
$E$ is uniformly contracted. In this case $F$ is the
\emph{central subbundle} of $\Lambda$.

A compact invariant set $\Lambda$ is said to be
\emph{singular-hyperbolic} if it is partially hyperbolic and
the action of the tangent cocycle expands volume along the
central subbundle, i.e.,
\begin{align}\label{eq:def-vol-exp}
      \vert \det (DX_t\vert_{F_x}) \vert > C e^{\la t},
      \forall t>0, \ \forall \ x \in \Lambda.
    \end{align}

The following definition was given as a particular case of singular hyperbolicity.

\begin{definition}\label{def:sec-exp}
  A \emph{sectional hyperbolic set} is a singular hyperbolic one such that
    for every two-dimensional linear subspace
   $L_x \subset F_x$ one has
    \begin{align}\label{eq:def-sec-exp}
      \vert \det (DX_t \vert_{L_x})\vert > C e^{\la t},
      \forall t>0.
    \end{align}
  \end{definition}

In \cite{Salg2016}, this author give another definition of singular hyperbolicity encompassing the two previous ones, as follow.

\begin{definition}\label{def:new-def-singhyp}
  A compact invariant set $\Lambda \subset M$ is
  \emph{$p$-sectional hyperbolic or singular hyperbolic of order $p$} for $X$ if
  all singularities in $\Lambda$ are hyperbolic, there
  exists a partially hyperbolic splitting of the tangent bundle on
  $T_{\Lambda}M = E \oplus F$ and constants $C,\lambda > 0$ such that for every $x
  \in \Lambda$ and every $t>0$ we have
\begin{enumerate}
\item $\Vert DX_t\vert_{E_x} \Vert \leq C e^{- \lambda t}$;
\item $\vert \wedge^p DX_t \vert_{L_x}\vert > C^{-1} e^{\la t}$,
  for every $p$-dimensional linear subspace $L_x \subset
  F_x$.
\end{enumerate}
\end{definition}

In our applications here, we only deal with two dimensional sectional hyperbolic case, but we conjecture that analogous results hold for sectional hyperbolic sets of any order $p$, with $2 \leq p \leq \dim{F}$.

From now on, we consider $M$ a connected compact finite
dimensional riemannian manifold and all singularities of
$X$ (if any) are hyperbolic.

\subsection{Fields of quadratic forms}
\label{sec:fields-quadrat-forms}

From now, we introduce the quadratic forms and its properties.

Let $E_U$ be a finite dimensional vector bundle with base $U$ and $\J:E_U\to\RR$ be a continuous family of quadratic forms $\J_x:E_x\to\RR$ which are non-degenerate and have index
$0<q<\dim(E)=n$, where $U\subset M$ is an open set such that $X_t(U) \subset \overline{U}, \forall t \geq 0,$
for a vector field $X$. We also assume that $(\J_x)_{x\in U}$ is continuously differentiable along the flow.

The continuity assumption on $\J$ just means that for every
continuous section $Z$ of $E_U$ the map $U\to\RR$ given by
$x\mapsto \J(Z(x))$ is continuous. The $C^1$ assumption on
$\J$ along the flow means that the map $x\mapsto
\J_{X_t(x)} (Z(X_t(x)))$ is continuously differentiable for
all $x\in U$ and each $C^1$ section $Z$ of $E_U$.

The assumption that $M$ is a compact manifold enables us to
globally define an inner product in $E$ with respect to
which we can find an orthonormal basis associated to
$\J_x$ for each $x$, as follows.

Fixing an orthonormal basis
on $E_x$ we can define the linear operator
\begin{align*}
  J_x:E_x\to E_x \quad\text{such that}\quad \J_x(v)=<J_x
  v,v> \quad \text{for all}\quad v\in E_x,
\end{align*}
where $<,>=<,>_x$ is the inner product at $E_x$. Since we
can always replace $J_x$ by $(J_x+J_x^*)/2$ without changing
the last identity, where $J_x^*$ is the adjoint of $J_x$
with respect to $<,>$, we can assume that $J_x$ is
self-adjoint without loss of generality.  Hence, we
represent $\J(v)$ by a non-degenerate symmetric bilinear
form $<\J_x v,v>_x$. Now we use Lagrange's method to
diagonalize this bilinear form, obtaining a base
$\{u_1,\dots,u_n\}$ of $E_x$ such that
\begin{align*}
  \J_x(\sum_{i}\alpha_iu_i)=\sum_{i=1}^q -\lambda_i\alpha_i^2 +
  \sum_{j=q+1}^n \lambda_j\alpha_j^2, \quad
  (\alpha_1,\dots,\alpha_n)\in\RR^n.
\end{align*}
Replacing each element of
this base according to $v_i=|\lambda_i|^{1/2}u_i$ we deduce that
\begin{align*}
\J_x(\sum_{i}\alpha_iv_i)=\sum_{i=1}^q -\alpha_i^2 +
  \sum_{j=q+1}^n \alpha_j^2, \quad
  (\alpha_1,\dots,\alpha_n)\in\RR^n.
\end{align*}
Finally, we can redefine $<,>$ so that the base
$\{v_1,\dots, v_n\}$ is orthonormal. This can be done
smoothly in a neighborhood of $x$ in $M$ since we are
assuming that the quadratic forms are non-degenerate; the
reader can check the method of Lagrange in a standard Linear
Algebra textbook and observe that the steps can be performed
with small perturbations, for instance in \cite{Maltsev63}.

In this adapted inner product we have that $J_x$ has entries
from $\{-1,0,1\}$ only, $J_x^*=J_x$ and also that
$J_x^2=J_x$.

Having fixed the orthonormal frame as above, the
\emph{standard negative subspace} at $x$ is the one spanned
by $v_{1},\dots, v_{q}$ and the \emph{standard positive
  subspace} at $x$ is the one spanned $v_{q+1},\dots,v_n$.

\subsubsection{Positive and negative cones}
\label{sec:positive-negative-co}

Let $\cC_\pm=\{C_\pm(x)\}_{x\in U}$ be the family of
positive and negative cones
\begin{align*}
  C_\pm(x):=\{0\}\cup\{v\in E_x: \pm\J_x(v)>0\}  \quad x\in U
\end{align*}
and also let $\cC_0=\{C_0(x)\}_{x\in U}$ be the corresponding
family of zero vectors $C_0(x)=\J_x^{-1}(\{0\})$ for all
$x\in U$.
In the adapted coordinates obtained above we have
\begin{align*}
  C_0(x)=\{v=\sum_{i}\alpha_iv_i\in E_x :
  \sum_{j=q+1}^n \alpha_j^2 = \sum_{i=1}^q
  \alpha_i^2\}
\end{align*}
is the set of \emph{extreme points} of $C_\pm(x)$.

The following definitions are fundamental to state our main
result.

\begin{definition}
\label{def:J-separated}
Given a continuous field of non-degenerate quadratic forms
$\J$ with constant index on the trapping region $U$ for the
flow $X_t$, we say that the flow is
\begin{itemize}
\item $\J$-\emph{separated} if $DX_t(x)(C_+(x))\subset
  C_+(X_t(x))$, for all $t>0$ and $x\in U$;
\item \emph{strictly $\J$-separated} if $DX_t(x)(C_+(x)\cup
  C_0(x))\subset C_+(X_t(x))$, for all $t>0$ and $x\in U$;
\item $\J$-\emph{monotone} if $\J_{X_t(x)}(DX_t(x)v)\ge \J_x(v)$, for each $v\in
  T_xM\setminus\{0\}$ and $t>0$;
\item \emph{strictly $\J$-monotone} if $\partial_t\big(\J_{X_t(x)}(DX_t(x)v)\big)\mid_{t=0}>0$,
  for all $v\in T_xM\setminus\{0\}$, $t>0$ and $x\in U$;
\item $\J$-\emph{isometry} if $\J_{X_t(x)}(DX_t(x)v) = \J_x(v)$, for each $v\in T_xM$ and $x\in U$.
\end{itemize}
\end{definition}
Thus, $\J$-separation corresponds to simple cone invariance
and strict $\J$-separation corresponds to strict cone
invariance under the action of $DX_t(x)$.

\begin{remark}\label{rmk:J-separated-C-}
  If a flow is strictly $\J$-separated, then for $v\in T_xM$
  such that $\J_x(v)\le0$ we have
  $\J_{X_{-t}(x)}(DX_{-t}(v))<0$ for all $t>0$ and $x$ such
  that $X_{-s}(x)\in U$ for every $s\in[-t,0]$. Indeed,
  otherwise $\J_{X_{-t}(x)}(DX_{-t}(v))\ge0$ would imply
  $\J_x(v)=\J_x\big(DX_t(DX_{-t}(v))\big)>0$, contradicting
  the assumption that $v$ was a non-positive vector.

  This means that a flow $X_t$ is strictly
    $\J$-separated if, and only if, its time reversal
    $X_{-t}$ is strictly $(-\J)$-separated.
\end{remark}

A vector field $X$ is $\J$-\emph{non-negative} on $U$ if
$\J(X(x))\ge0$ for all $x\in U$, and
$\J$-\emph{non-positive} on $U$ if $\J(X(x))\leq 0$ for all
$x\in U$. When the quadratic form used in the context is
clear, we will simply say that $X$ is non-negative or
non-positive.

We apply this notion to the linear Poincar\'e flow defined on
regular orbits of $X_t$ as follows.

Suppose that the vector field $X$ is non-negative on $U$.
Then, the span $E^X_x$ of $X(x)\neq 0$ is a $\J$-non-degenerate subspace.

According to item (1) of Proposition~\ref{pr:propbilinear}, we have
that $T_xM=E_x^X\oplus N_x$, where $N_x$ is the
pseudo-orthogonal complement of $E^X_x$ with respect to the
bilinear form $\J$, and $N_x$ is also
non-degenerate. Moreover, by the definition, the index of $\J$
restricted to $N_x$ is the same as
the index of $\J$.
Thus, we can define on $N_x$ the positive and negative cones with core $N_x^+$ and $N_x^-$, respectively.

Define the Linear Poincar\'e Flow $P^{\, t}$ of $X_t$
along the orbit of $x$, by projecting $DX_t$ orthogonally
(with respect to $\J$) over $N_{X_t(x)}$ for each $t\in\RR$:
\begin{align*}
  P^{\, t} v := \Pi_{X_t(x)}DX_t v,\\
  \ v\in T_x M, \ t\in\RR, X(x) \neq 0,
\end{align*}
where $\Pi_{X_t(x)}:T_{X_t(x)}M\to N_{X_t(x)}$ is the
projection on $N_{X_t(x)}$ parallel to $X(X_t(x))$.  We
remark that the definition of $\Pi_x$ depends on $X(x)$ and
$\J_X$ only. The linear Poincar\'e flow $P^{\,t}$ is a linear
multiplicative cocycle over $X_t$ on the set $U$ with the
exclusion of the singularities of $X$.

In this setting we can say that the linear Poincar\'e flow is
(strictly) $\J$-separated and (strictly) $\J$-monotonous
using the non-degenerate bilinear form $\J$ restricted to
$N_x$ for a regular $x\in U$.

More precisely: $P^t$ is $\J$-monotonous if $\partial_t\J(P^tv)\mid_{t=0}\ge0$, for
each $x\in U, v\in T_xM\setminus\{0\}$ and $t>0$, and
strictly $\J$-monotonous if $\partial_t\J(P^tv)\mid_{t=0}>0$,
for all $v\in T_xM\setminus\{0\}$, $t>0$ and $x\in U$.

\begin{proposition}\cite{Wojtk01}
  \label{pr:J-separated-spectrum}
  Let $L:V\to V$ be a $\J$-separated linear operator. Then
  \begin{enumerate}
  \item $L$ can be uniquely represented by $L=RU$, where $U$
    is a $\J$-isometry and
    $R$ is $\J$-symmetric (or $\J$-pseudo-adjoint; see
    Proposition~\ref{pr:propbilinear}) with positive
    spectrum.
  \item the operator $R$ can be diagonalized by a
    $\J$-isometry. Moreover the eigenvalues of $R$ satisfy
    \begin{align*}
      0<r_-^q\le\dots\le r_-^1=r_-\le r_+=r_1^+\le\dots\le r_+^p.
    \end{align*}
  \item the operator $L$ is (strictly) $\J$-monotonous if,
    and only if, $r_-\le (<) 1$ and $r_+\ge (>) 1$.
  \end{enumerate}
\end{proposition}

\subsection{$\J$-separated linear maps}
\label{sec:j-separat-linear}

\subsubsection{$\J$-symmetrical matrixes and $\J$-selfadjoint operators}
\label{sec:j-symmetr-matrix}

The symmetrical bilinear form defined by $(v,w)=\langle J_x
v,w\rangle$, $v,w\in E_x$ for $x\in M$ endows
$E_x$ with a pseudo-Euclidean structure. Since $\J_x$ is
non-degenerate, then the form $(\cdot,\cdot)$ is likewise
non-degenerate and many properties of inner products are
shared with symmetrical non-degenerate bilinear forms. We
state some of them below.

We recall that $E^\perp:=\{v\in V: (v,w)=0
    \quad\text{for all}\quad w\in E\}$, the
    pseudo-orthogonal space of $E$, is defined using the
    bilinear form.
\begin{proposition}\cite{Maltsev63}
  \label{pr:propbilinear}
  Let $(\cdot,\cdot):V\times V \to\RR$ be a real symmetric
  non-degenerate bilinear form on the real finite
  dimensional vector space $V$.
  \begin{enumerate}
  \item $E$ is a subspace of $V$ for which $(\cdot,\cdot)$ is
    non-degenerate if, and only if, $V=E\oplus E^\perp$.
  \item Every base $\{v_1,\dots,v_n\}$ of $V$ can be
    orthogonalized by the usual Gram-Schmidt process of
    Euclidean spaces, that is, there are linear combinations
    of the basis vectors $\{w_1,\dots, w_n\}$ such that they
    form a basis of $V$ and
    $(w_i,w_j)=0$ for $i\neq j$.  Then this last base can be
    pseudo-normalized: letting $u_i=|(w_i,w_i)|^{-1/2}w_i$ we
    get $(u_i,u_j)=\pm\delta_{ij}, i,j=1,\dots,n$.
  \item There exists a maximal dimension $p$ for a subspace
    $P_+$ of $\J$-positive vectors and a maximal dimension
    $q$ for a subspace $P_-$ of $\J$-negative vectors;
    we have $p+q=\dim V$ and $q$ is known
    as the \emph{index} of $\J$.
  \item For every linear map $L:V\to\RR$ there exists a
    unique $v\in V$ such that $L(w)=(v,w)$ for each $w\in V$.
  \item For each $L:V\to V$ linear there exists a unique
    linear operator $L^+:V\to V$ (the pseudo-adjoint) such that
    $(L(v),w)=(v,L^+(w))$ for every $v,w\in V$.
  \item Every pseudo-self-adjoint $L:V\to V$, that is,
    such that $L=L^+$, satisfies
    \begin{enumerate}
    \item eigenspaces corresponding to distinct eigenvalues
      are pseudo-orthogonal;
    \item if a subspace $E$ is $L$-invariant, then $E^\perp$
      is also $L$-invariant.
    \end{enumerate}
  \end{enumerate}
\end{proposition}

The proofs are rather standard and can be found in
\cite{Maltsev63}.

The following simple result will be very useful in what follows.

\begin{lemma}\cite[Theorem 1.2]{Wojtk01}
  \label{le:kuhne}
Let $V$ be a real finite dimensional vector space endowed
with a non-positive definite and non-degenerate quadratic
form $\J:V\to\RR$.

If a symmetric bilinear form $F:V\times V\to\RR$ is
non-negative on $C_0$ then
\begin{align*}
  r_+=\inf_{v\in C_+} \frac{F(v,v)}{\langle Jv,v\rangle}
  \ge \sup_{u\in C_-}\frac{F(u,u)}{\langle Ju,u\rangle}=r_-
\end{align*}
and for every $r$ in $[r_-,r_+]$ we have
$F(v,v)\ge r\langle Jv,v\rangle$ for each vector $v$.

In addition, if $F(\cdot,\cdot)$ is positive on
$C_0\setminus\{0\}$, then $r_-<r_+$ and $F(v,v)>
r\langle Jv,v\rangle$ for all vectors $v$ and $r\in(r_-,r_+)$.
\end{lemma}

\begin{remark}
  \label{rmk:Jseparated}
  Lemma~\ref{le:kuhne} shows that if $F(v,w)=\langle \tilde J
  v,w\rangle$ for some self-adjoint operator $\tilde J$ and
  $F(v,v)\ge0$ for all $v$ such that $\langle J v,
  v\rangle=0$, then we can find $a\in\RR$ such that
  $\tilde J \ge a J$. This means precisely that $\langle
  \tilde J v,v\rangle\ge a\langle Jv, v\rangle$ for all
  $v$.

  If, in addition, we have $F(v,v)>0$ for all $v$ such that
  $\langle J v, v\rangle=0$, then we obtain a strict
  inequality $\tilde J > a J$ for some $a\in\RR$ since the
  infimum in the statement of Lemma~\ref{le:kuhne} is
  strictly bigger than the supremum.
\end{remark}

The (longer) proofs of the following results can be found
in~\cite{Wojtk01} or in~\cite{Pota79}; see also~\cite{Wojtk09}.

For a $\J$-separated operator $L:V\to V$ and a
$d$-dimensional subspace $F_+\subset C_+$, the subspaces
$F_+$ and $L(F_+)\subset C_+$ have an inner product given by
$\J$. Thus both subspaces are endowed with volume
elements. Let $\alpha_d(L;F_+)$ be the rate of expansion of
volume of $L\mid_{F_+}$ and $\sigma_d(L)$ be the infimum of
$\alpha_d(L;F_+)$ over all $d$-dimensional subspaces $F_+$
of $C_+$.

\begin{proposition}\cite[Proposition 1.3]{Wojtk01}
  \label{pr:product-vol-exp}
  We have $\sigma_d(L)=r_+^1 \cdots r_+^d$, where $r^i_+$
  are given by Proposition~\ref{pr:J-separated-spectrum}.
\end{proposition}
Moreover, is easy to see that, if $L_1,L_2$ are $\J$-separated, then
  $\sigma_d(L_1L_2)\ge\sigma_d(L_1)\sigma_d(L_2)$.

The following corollary is very useful.

\begin{corollary} \cite[Corollary 1.5]{Wojtk01}
  \label{cor:compos-max-exp}
  For $\J$-separated operators $L_1,L_2:V\to V$ we have
  \begin{align*}
    r_+^1(L_1L_2)\ge r_+^1(L_1) r_+^1(L_2) \quad\text{and}\quad
    r_-^1(L_1L_2)\le r_-^1(L_1)r_-^1(L_2).
  \end{align*}
  Moreover, if the operators are strictly $\J$-separated,
  then the inequalities are strict.
\end{corollary}

\begin{remark}\label{rmk:J-mon-spec}
Another important property about the singular values of a $\J$-separated operator $L$ is that
$$  r_+^1 = r_+ \ge 1 (> 1) \quad\text{and}\quad r_-^1 = r_- \le 1 (< 1)$$
if, and only if, $L$ is (strictly) $\J$-monotone.

The last property will be used a lot of times in our proofs.
\end{remark}

\subsubsection{Lyapunov exponents}

By Oseledec's Ergodic Theorem \cite{Os68}, there exist a full probability set $X$ such that for every $x \in Y$ there is an invariant decomposition
\begin{align*}
T_xM = \langle X\rangle \oplus E_{1}(x) \oplus \cdots \oplus E_{l(x)}(x)
\end{align*}
and numbers $\chi_1 < \cdots < \chi_l$ corresponding to the limits
\begin{align*}
  \chi_j = \lim\limits_{t \to +\infty} \frac{1}{t} \log \Vert DX_t(x) \cdot v\Vert,
\end{align*}
for every $v \in E_i(x)\setminus \{0\}, i = 1, \cdots, l(x)$.

In this setting, Wojtkowski \cite{Wojtk01} proved that the logarithm of the pseudo-Euclidean singular values $0 \leq r_q^- \leq \cdots \leq r_1^- \leq r_1^+ \leq \cdots \leq r_p^+$ of $DX_t$ are $\mu$-integrable, and obtained estimates of the Lyapunov exponents related to the singular eigenvalues of strictly $\J$-separated maps.

\vspace{0.1in}

\begin{theorem}\cite[Corollary 3.7]{Wojtk01}
\label{thm:lyap-exp-sing-val}
For $1 \leq k_1 \leq p$  and $1 \leq k_2 \leq q$

\begin{align*}
\chi^-_1 + \cdots + \chi^-_{k_1} \leq \sum_{i=1}^{k_1} \int \log r^-_i d\mu \ \textrm{and} \ \chi^+_1 + \cdots + \chi^+_{k_2} \geq \sum_{i=1}^{k_2} \int \log r^+_i d\mu.
\end{align*}

\end{theorem}

This result will be very useful in proof of Theorem \ref{mthm:strong-homog-to-quad}.

%\newpage

\section{Some applications}

In this section, we present some applications of this theory related to some kind of hyperbolicities, as partial and singular ones.
In particular, we provide another proof of \cite[Theorem A]{GanLiWen2005}.

\subsubsection{Some results about partial and sectional hyperbolicity from $\J$-separation}

The author, together with V. Ara\'ujo, proved in \cite{ArSal2012} the following useful theorem which relates partial hyperbolicity and $\J$-separated sets for a flow.

\begin{theorem}\cite[Theorem A]{ArSal2012}
  \label{mthm:Jseparated-parthyp}
  A maximal invariant subset $\Lambda$ of a trapping region
  $U$ whose singularities are hyperbolic
  is a partially hyperbolic set for a flow $X_t$ if, and
  only if, there is a $C^1$ field $\J$ of non-degenerate
  quadratic forms with constant index, equal to the
  dimension of the stable subspace of $\Lambda$, such that
  $X_t$ is a non-negative strictly $\J$-separated flow on
  $U$.
\end{theorem}

This result will be useful in our applications of Theorem \ref{mthm:strong-homog-equiv}.

In the sequence, we can give another proof of next result from \cite{GanLiWen2005}.

\begin{theorem}\cite[Theorem A]{GanLiWen2005} Let $X \in \mathfrak{X}^{1}(M)$, and $\Lambda$
be a robustly transitive singular set of $X$ that
is strongly homogeneous of index $\indi$. If every singularity $\sigma$ of $X$ is hyperbolic
of index $\indi(\sigma) > \indi$, then $\Lambda$ has a partially hyperbolic splitting of contracting
dimension Ind. Likewise, if every singularity $\sigma$ of $X$ is hyperbolic of index
$\indi(\sigma) \leq \indi$, then $\Lambda$ has a partially hyperbolic splitting of expanding dimension
$n - 1 - Ind$.
\end{theorem}

\proof
We are going to deal with the case $\indi(\sigma) > \indi$, the other case is analogous.

Since $\Lambda$ is strongly homogeneous and $\indi(\sigma) > \indi$, by \cite[Lemma 4.1]{GanLiWen2005} there is a dominated splitting $T_{\sigma}M = E_{\sigma} \oplus F_{\sigma}$ such that $\dim(E_{\sigma}) = \indi$.
Hence, Theorem \ref{mthm:strong-homog-to-quad} implies that there exists a field of non-degenerate quadratic forms $\J$ on $\Lambda$ with index $\indi(\J) = \indi(\Lambda)$ for which $X$ is %non-negative
strictly $\J$-separated and the associated linear Poincar\'e flow $P^t$ is strictly $\J$-monotone on every compact invariant subset $\gamma$ of $\Lambda^*$.
Therefore, Theorem \ref{mthm:Jseparated-parthyp} completes the proof.$\diamond$

Some immediate results follow from the main theorems.

The following consequences of these results follows from the
robustness of sectional hyperbolicity and the theory of
sectional hyperbolic transitive sets for homogeneous flows
from~\cite{MeMor06} and \cite{ArbMo2013}.

\begin{corollary}\label{mcor:2-sec-exp-J-monot}
  Let $X\in \mathfrak{X}^{1}(M), \dim(M) \geq 4$ with a nontrivial transitive compact invariant set
  $\Lambda$ whose singularities, if any, are
  hyperbolic.

  Then the following conditions are equivalent:
  \begin{enumerate}
  \item There exists a family $\J$ of smooth non-degenerate
    indefinite quadratic forms with constant index
    $\indi(\J)$ on $\Lambda$ such that $X$ is a
    non-negative strictly $\J$-separated vector field, for which the linear
    Poincar\'e flow is strictly $\J$-monotonous on every compact
    invariant set in $\Lambda_X(U)^*=\Lambda_X(U)\setminus \sing(X)$
  \item The set $\Lambda$ is a sectional-hyperbolic
    subset for $X$ with constant index
    $\indi(\cO)=\indi(\J)$ for all periodic orbits $\cO$ of
    $\Lambda$ and $\indi(\sigma)=Ind(\J)+1$ for all
    singularities $\sigma\in\Lambda\cap \sing(X)$.
\end{enumerate}
\end{corollary}

For the next statement, we recall that a hyperbolic singularity
$\sigma$ is said to be of codimension one if its index satisfies
either $\indi(\sigma) = 1$ or $\indi(\sigma) = n - 1$, where $n = \dim(M)$.

\begin{remark}\label{obs-lyap-est}
Every attracting set is Lyapunov stable.
\end{remark}

\begin{corollary}\label{thm:lyap-est-sec-hyp}
 Let $\Lambda \subset M^n, n \geq 4$, be a nontrivial transitive set, which is Lyapunov stable for $X$, with singularities all of them hyperbolic of codimension one.
Then, the following properties are equivalent:
\begin{enumerate}
\item{} $\Lambda$ is sectional-hyperbolic with $1 \leq dim(E^s) = \indi(\J) \leq n-2$;
\item{} There exists a field of non-degenerate quadratic forms with
constant index $1 \leq \indi(\J) \leq n-2$ such that $X$ is non-negative
strictly $\J$-separated on $\Lambda$ and every compact invariant subset
$\Gamma \subset \Lambda$ is strictly $\J$-monotone for linear Poincar\'e flow associated to $X$.
\end{enumerate}
\end{corollary}

\subsection{Proof of Corollaries \ref{mcor:2-sec-exp-J-monot} and \ref{thm:lyap-est-sec-hyp}}

\proof[proof of Corollary \ref{mcor:2-sec-exp-J-monot}]

First of all, we will need the following lemma from \cite{ArbMo2013}:

\begin{lemma}\label{le:cor8-ArbMor}\cite[Corollary 8]{ArbMo2013}
Let $\Lambda$ be a nontrivial transitive set of $X$ which is strongly homogeneous with singularities (all of them of codimension one).
If $n\geq 4$ and $1 \leq \indi(\Lambda) \leq n-2$, then $\Lambda$ is sectional hyperbolic up to a flow-reversing.
\end{lemma}

 Indeed, suppose that $(2)$ is true. Then, $X$ is strongly homogeneous on $\Lambda$.
 By Lemma \ref{le:cor8-ArbMor} this is a sectional hyperbolic set of $X$.
 To prove the converse statement, we need just use \cite[Theorem D]{ArSal2012}.$\diamond$

The next proof needs the following lemmas.

Let $\Lambda$ be a compact invariant set for a flow
  $X$ of a $C^1$ vector field $X$ on $M$.
\begin{lemma} \cite[Lemma 5.1]{AraArbSal}
  \label{le:flow-center}
  Given a continuous splitting $T_\Lambda M = E\oplus F$
  such that $E$ is uniformly contracted, then $X(x)\in F_x$ for all $x\in \Lambda$.
\end{lemma}
and
\begin{lemma} \cite[Corollary 9]{ArbMo2013}
  \label{le:ArbMo13}
  Let $\Lambda$ be a nontrivial transitive set of $X$ which is strongly homogeneous with singularities (all of them hyperbolic of codimension one).
  If $n \geq 4$, $1\geq \indi(\Lambda)\geq n-2$ and $\Lambda$ is Lyapunov stable, then it is sectional hyperbolic for $X$.
\end{lemma}

\proof[Proof of Corollary \ref{thm:lyap-est-sec-hyp}]

Suppose that $\Lambda$ is sectional-hyperbolic with decomposition $E \oplus F$. So, it is clearly strongly homogeneous. Once the subbundles are non-trivial and $E$ is uniformly contracting, we must have $1 \leq \dim(E) := \indi(\J) \leq n-2$, because by Lemma \ref{le:flow-center}, $\langle X \rangle \subset F$.

 By Theorem \ref{mthm:Jseparated-parthyp}, there exists a field $\J$ of differentiable quadratic forms with constant index equal to the dimension of $E$ with the required properties.

  Reciprocally, the existence of such a field $\J$ implies, by Theorem \ref{mthm:Jseparated-parthyp},
  that $\Lambda$ is strongly homogeneous of index $\indi(\J)$.
  Thus, once the singularities are hyperbolic of codimension one, it is enough to use Lemma \ref{le:ArbMo13}.$\diamond$

\section{Proof of Theorems \ref{cor:star-flow}, \ref{mthm:strong-homog-equiv} and Corollary \ref{mcor:strong-homog}}

Now, we prove our mains results.

Before to proceed the proofs, we need to recall some results related to the theory of the Extended Linear Poincar\'e Flow from Gan, Li and Wen \cite{GanLiWen2005}.

\subsection{Extended Linear Poicar\'e Flow}\label{sec:ELPF}

First of all, we recall that the Linear Poincar\'e Flow $P^{\, t}$ of $X_t$
along the orbit of $x$ is defined by projecting $DX_t$ orthogonally
over $N_{X_t(x)}$ for each $t\in\RR$:
\begin{align*}
  P^{\, t} v := \Pi_{X_t(x)}DX_t v ,
  \quad
  v\in T_x M, t\in\RR, X(x)\neq 0,
\end{align*}
where $\Pi_{X_t(x)}:T_{X_t(x)}M\to N_{X_t(x)}$ is the
projection on $N_{X_t(x)}$ parallel to $X(X_t(x))$.  We
remark that the definition of $\Pi_x$ depends on $X(x)$ only.
The linear Poincar\'e flow $P^{\,t}$ is a linear
multiplicative cocycle over $X_t$ on the set $U$ with the
exclusion of the singularities of $X$.

It is well known \cite{Do87} that a vector field on $M$ is Anosov if, and only if, the associated Linear Poincar\'e Flow is hyperbolic on $N_x$, for all $x \in M$. In this case, it is also well known that there is no singularities of $X$ in $M$. A similar result holds to a compact $X$-invariant set $\Lambda \subset M$, changing "Anosov" by "uniformly hyperbolic" in the statement.

In the context of singular flows presenting certain kind of (weak) hyperbolicity, the authors in \cite{GanLiWen2005} worked out an extension of the classical Linear Poincar\'e Flow to singularities. As it is used in some results contained in this paper, here we present some proofs for completeness.

Denote the one-dimensional Grassmannian manifold of $M$ by
$$
G^1 = G^1(M) = \{L ; L \ \textrm{is an one dimensional subspace of} \ T_xM, x \in M\}.
$$

Given a $C^1$ vector field $X$, from the tangent flow $DX_t$ we may induce a flow on $G^1$
\begin{eqnarray*}
%\begin{array}
\Phi_t: G^1 \to G^1 \\
 L \mapsto DX_t(L) .
%\end{array}
\end{eqnarray*}

Based on \cite{GanLiWen2005}, we define the extended linear Poincar\'e flow as follow.

Let $\beta: G^1 \to M$ and $\zeta: TM \to M$ denote the two bundle projections of $G^1$ and $TM$ on $M$, respectively.
Denote the pullback bundle by
$$
\beta^*(TM) = \{(L,v) \in G^1 \times TM; \beta(L) = \zeta(v)\}.
$$
This is a $n$-dimensional vector bundle over the base space $G^1$ with the bundle projection
\begin{eqnarray*}
\iota : \beta^*(TM) \to G^1 \\
\iota (L,v) = L.
\end{eqnarray*}

The induced inner product on $\beta^*(TM)$ is given by
$$
\langle (L,u),(L,w) \rangle =\langle u,w \rangle
$$
for any $(L,u),(L,w) \in \beta^*(TM)$, where $\langle \cdot , \cdot \rangle$ is the inner product on $TM$.

The tangent flow $DX_t:TM \to TM$ induces an \emph{extended tangent flow} on $\beta^*(TM)$
\begin{eqnarray*}
\Phi^{*}_{t}: \beta^*(TM) \to \beta^*(TM) \\
\Phi^{*}_{t}(L, v) = (\Phi_{t}(L), DX_t(v))
\end{eqnarray*}
such that the diagram below
$$
\begin{array}{ccc}
\beta^*(TM) & \stackrel {\Phi^{*}_{t}}{\longrightarrow} & \beta^*(TM)\\
\iota \downarrow  &  &  \downarrow \iota \\
G^1 & \stackrel {\Phi_t}{\longrightarrow} & G^1
\end{array}
$$
commutes.

Given any $L \in G^1$ and any subspace $E \subset T_{\beta(L)}M$, by the relation between $DX_t$ and its induced flow on $\beta^*(TM)$, we have
$$
\Vert \Phi^*_t\mid_{\{L\} \times E} \Vert = \Vert DX_t\mid_{E} \Vert.
$$

We also define the following subbundles of $\beta^{*}(TM)$

\begin{align*}
S = \{(L, v) \in\beta^{*}(TM); v \in L\}
\end{align*}
which is an one-dimensional subbundle of $\beta^*(TM)$, independent of $X$,
and its orthogonal complement

\begin{align*}
\mathcal{N} = \{(L, v) \in\beta^{*}(TM); v \perp L\}.
\end{align*}

The fiber of $S$ at $L$, $P_L$, is the set of vectors that represent the direction vectors of $L$.

The extended Poincar\'e flow $\Psi^{X}_t:\mathcal{N}^X \to \mathcal{N}^X $ is defined by

\begin{align}\label{def-ext-poinc-flow}
\Psi^{X}_t = \Pi^X \circ \Phi^{*,X}_t,
\end{align}

where $\Pi: \beta^{*}(TM) \to \mathcal{N}$ is the orthogonal projection.

A dominated splitting for $\Psi^{X}_t$
over a subset $A \subset G1$ consists of a
direct sum of $\Psi^{X}_t$-invariant subbundles
$\mathcal{N}^Y_A = N^1_A \oplus N^2_A$ together with positive
constants $K, \lambda$ satisfying
$$
\frac{\Vert \Psi^{X}_t\vert_{N^1_L} \Vert}{m(\Psi^{X}_t\vert_{N^2_L} )}\leq K e^{-\lambda t}, \forall \ (L, t) \in A \times \mathbb{R}^+.
$$

The useful lemma below will be used soon.

\begin{lemma}\cite[Lemma 3.3]{GanLiWen2005}
\label{le:GLW}

Let $X \in \mathfrak{X}^{1}(M)$ and $\sigma \in Sing(X)$. Let $B \in G^1_{\sigma}$ be an invariant set
of $\Phi_t:G^1 \to G^1$ and $E \subset T_{\sigma}M$ be an invariant linear subspace of $X_t: TM \to TM$.
\begin{enumerate}
\item{} If $E = E^1 \oplus E^2$ is a dominated splitting for $X_t$, then $\beta^*(E)\mid_B = \beta^*(E^1)\mid_B \oplus \beta^*(E^2)\mid_B$ is a dominated splitting for $\Phi^*_t$.
\item{} If $\beta^*(E)\mid_B = F^1 \oplus F^2$ is a dominated splitting for $\Phi^*_t$, then $F^1_L$ and $F^2_L$ are independent of $L \in B$. In fact, there is a dominated splitting $E = E^1 \oplus E^2$ for $X_t$, such that $F^i = \beta^*(E^i)\mid_B, i=1,2$.
\end{enumerate}
\end{lemma}

The next lemma will be very useful in the sequel, and the original statement can be found at \cite[Lemma 4.1]{GanLiWen2005}. There the authors also assume the strong transitiveness of the vector field, but we stress that the same proof holds without it. For completeness, its proof is presented here.

For any $\sigma \in \sing(X) \cap \Lambda$ we denote $B_{\sigma}(\Lambda) = \{ L \in B(\Lambda): \beta(L) = \sigma\}$, where $B(\Lambda)= \{L \in G^1: \beta(L) \in \Lambda \cap Per_*(X), i.e., \exists Y_n \to X, p_n \in Per(Y_n), \textrm{such \ that} \langle Y_n(p_n)\rangle \to L\}$.

\begin{lemma}
\label{le:GLW2}

Let $X \in \mathfrak{X}^{1}(M)$ and $\Lambda$ be a singular set of $X$ which is strongly homogeneous of index $\indi$. Suppose that $\sigma \in Sing(X) \cap \Lambda$, with $\indi(\sigma) > \indi$. Then $E^{s}_{\sigma}$ splits into a dominated splitting $E^{s}_{\sigma} = E^{ss}_{\sigma} \oplus E^{c}_{\sigma}$ with respect to $X_t: TM \to TM$ where $\dim E^{ss}_{\sigma} =  \indi$, such that for any $L \in B_\sigma (\Lambda)$ with $L \in E^u_\sigma$ one has $N^{s}_{L} = \{L\} \times E^{ss}_{\sigma}$.
\end{lemma}

\proof
  Consider
  \begin{align*}
    B^u_\sigma (\Lambda) = \{L \in B_\sigma (\Lambda): L \in E^u_\sigma\}.
  \end{align*}
  We have $B^u_\sigma (\Lambda) \neq \emptyset$, closed and $\Phi_t$-invariant. By Lemma \ref{le:GLW}, there is a dominated splitting on $N^{u}_{B_\sigma (\Lambda)}$ with respect to $\Phi_t^*$ of index $\indi$:
    \begin{align}\label{eq:dom-spl-normal}
   N_{B^u_\sigma (\Lambda)} = N^{s}_{B^u_\sigma (\Lambda)} \oplus N^{u}_{B^u_\sigma (\Lambda)}.
  \end{align}
  And, then, there exists a dominated splitting, with respect to $\Phi_t^*$,
  \begin{align}\label{eq:dom-spl-normal-2}
   \beta^*(TM)\vert_{B^u_\sigma (\Lambda)} = \beta^*(E^s_\sigma)\vert_{B^u_\sigma (\Lambda)} \oplus \beta^*(E^u_\sigma)\vert_{B^u_\sigma (\Lambda)}.
  \end{align}
  By intersecting with $N_{B^u_\sigma (\Lambda)}$, and observing that $E^s_\sigma \perp E^u_\sigma$, we obtain another spltting
\begin{align}\label{eq:dom-spl-norm-3}
N_{B^u_\sigma (\Lambda)} = \beta^*(E^s_\sigma)\vert_{B^u_\sigma (\Lambda)} \oplus (\beta^*(E^u_\sigma)\vert_{B^u_\sigma (\Lambda)} \cap N_{B^u_\sigma (\Lambda)})
\end{align}
such way that, for any $L \in B^u_\sigma (\Lambda)$,
\begin{align*}
N_L = \{L\} \times L^\perp = \{L\} \times E^s_\sigma \oplus \{L\} \times (E^u_\sigma \cap L^\perp),
\end{align*}
where $L^\perp$ is the orthogonal complement subspace of $L$ in $T_\sigma M$.

As $E^s_\sigma \perp L$, for any $L \in B^u_\sigma (\Lambda)$, we have this splitting continuous and invariant under $\Phi_t: N \to N$.

{\b Claim}: This splitting is dominated under $\Phi_t: N \to N$.

Indeed, since $\sigma$ is hyperbolic, there is a $T > 0$ such that
\begin{align*}
  \Vert X_T\vert_{E^s_\sigma}\Vert \cdot\Vert X_{-T}\vert_{E^u_\sigma}\Vert \leq \frac{1}{2}.
\end{align*}
And, since $E^s_\sigma \perp L$, for any $L \in B^u_\sigma (\Lambda)$, we have
\begin{align*}
  \Phi_t\vert_{\{L\} \times E^s_\sigma} =  \Phi^*_t\vert_{\{L\} \times E^s_\sigma}.
\end{align*}
Hence,
\begin{align*}
 \Vert \Phi_T\vert_{\{L\} \times E^s_\sigma} \Vert = \Vert \Phi^*_T\vert_{\{L\} \times E^s_\sigma}\Vert = \Vert X_T\vert_{E^s_\sigma}\Vert.
\end{align*}

On the other hand, $L \subset E^u_\sigma$ implies that
\begin{align*}
 \Vert \Phi_{-T}\vert_{\{L_T\} \times (E^u_\sigma \cap L^{\perp}_T)} \Vert \leq \Vert \Phi^*_{-T}\vert_{\{L_T\} \times E^u_\sigma}\Vert = \Vert X_{-T}\vert_{E^u_\sigma}\Vert,
\end{align*}
where $L_T = X_T(L)$. Then, we obtain
\begin{align*}
  \Vert \Phi_T\vert_{\{L\} \times E^s_\sigma} \Vert \cdot \Vert \Phi_{-T}\vert_{\{L_T\} \times (E^u_\sigma \cap L^{\perp}_T)} \Vert \leq \frac{1}{2}.
\end{align*}

So, (\ref{eq:dom-spl-normal-2}) is dominated for $\Phi_t:N \to N$.

By (\ref{eq:dom-spl-normal}) and (\ref{eq:dom-spl-normal-2}), and because $\indi(\sigma) > \indi$, we get
\begin{align*}
  N^{s}_{B^u_\sigma (\Lambda)}\subset \beta^*(E^s_\sigma)\vert_{B^u_\sigma (\Lambda)} \ \ \beta^*(E^u_\sigma)\vert_{B^u_\sigma (\Lambda)} \cap N^{u}_{B^u_\sigma (\Lambda)} \subset N^{u}_{B^u_\sigma (\Lambda)}.
\end{align*}

Moreover, we have that
\begin{align}\label{eq:dom-splt-ELPF}
  \beta^*(E^s_\sigma)\vert_{B^u_\sigma (\Lambda)} = N^{s}_{B^u_\sigma (\Lambda)}\subset \oplus (N^{u}_{B^u_\sigma (\Lambda)} \cap \beta^*(E^s_\sigma)\vert_{B^u_\sigma (\Lambda)})
\end{align}
is a dominated splitting under $\Phi_t:N \to N$. Finally, as $\Phi_t\vert_{\beta^*(E^s_\sigma)\vert_{B^u_\sigma (\Lambda)}} = \Phi^*_t\vert_{\beta^*(E^s_\sigma)\vert_{B^u_\sigma (\Lambda)}}$, the splitting \ref{eq:dom-splt-ELPF} is also dominated with respect to $\Phi^*_t:\beta^*(TM) \to \beta^*(TM)$.

By Lemma \ref{le:GLW}, $E^s_\sigma = E^{ss}_{\sigma} \oplus E^{c}_{\sigma}$ is a dominated splitting for $X_t:TM \to TM$, with $\dim E^{ss}_{\sigma} = \dim (N^{s}_{B^u_\sigma (\Lambda)}) = \indi$.

In addition, for any $L \in B_\sigma (\Lambda)$ such that $L \subset E^u_\sigma$, one has $N_L^s = \{L\} \times E^{ss}_{\sigma}$.

This complete the proof of our claim and the Lemma as well.$\diamond$

To prove the Theorem \ref{cor:star-flow} we use the following result from \cite{ArSal2012}.
\begin{proposition}\cite[item 3,Theorem 2.23]{ArSal2012}
\label{prop:J-hyperbolic}
Let $\Gamma$ be a compact invariant set for $X$ with a dominated splitting $T_{\Gamma}N = E \oplus F$ for the linear Poincar\'e flow $P^t(x)$ on $X$ over $\Gamma$. Let $\J$ be a $C^1$ field of indefinite quadratic forms such that $P^t(x)$ is strictly  $\J$-separated. Then, $E \oplus F$ is uniformly hyperbolic if, and only if, there is an equivalent field $\J$ of quadratic forms on a neighborhood of $\Gamma$ such that $\J'(v) > 0$, for all $v \in T_{\Gamma}N$ and all $x \in \Gamma$.
\end{proposition}

\begin{remark}\label{rmk:thm-Jhyp-poincare}
We stress that the cocycle $A_t(x)$ in the original statement of the above result has been replaced by the linear Poincar\'e flow $P^t$, in our applications here.
\end{remark}

\proof [Proof of Theorem \ref{cor:star-flow}]
If $X$ is a star flow, then each singular point $\sigma$ is hyperbolic and it is well known that its hyperbolic decomposition $E^s_{\sigma} \oplus E^u_{\sigma}$ is a dominated one. So, by using adapted metrics (see \cite{Goum07}) we construct the desired quadratic form $J_\sigma$ such that $X$ is strictly separated (see \cite{ArSal2012}) and, by Proposition \ref{prop:J-hyperbolic}  $\J'(v)>0$ for all $v \in T_{\Gamma}M$.

Analogously, for every periodic orbit $\gamma$ of $X$, consider the hyperbolic splitting $T_\gamma M = E^s \oplus E^X \oplus E^u$. Again, considering $E^s \oplus (E^X \oplus E^u)$ as a dominated splitting we obtain a quadratic form $\J$ for which $X$ is strictly separated on $\gamma$. By construction of the adapted metrics, we have that $\J$ is $C^ 1$ along the flow (see \cite{Goum07} for details about the construction of such a adapted metric). In addition, the linear Poincar\'e flow associated to $X$, $P^t$ is hyperbolic and then $\J$-monotone on $\gamma$.

If $\gamma$ is a sink (respectively, a source) the splitting $E^s \oplus E^X$ (respectively, $E^u \oplus E^X$) is a dominated one and we proceed constructing the cones the same way, however the core of the nonnegative cone is the field direction.

Moreover, by definition of a preperiodic point $p$, there is a sequence of vector fields $Y_n$ and periodic points $q_n$ of $Y_n$ such that $Y_n \to X$ in $C^1$ topology and $q_n \to p$. Up to a subsequence, we can assume that the orbits of $(q_n, Y_n)$ have the same stable index and tends, in Hausdorff sense, to a compact $X$-invariant set $L$, which is clearly contained into $P_*(X)$ and $p\in L$.

If $T_{y_n}N=N^s_n\oplus N^u_n$ is a $P^{\,t}_{Y_n}$-hyperbolic splitting,
which is equivalent to be strictly $\J_{n}$-monotonous, where $\J_n$ is the quadratic form corresponding to the $Y_n$-hyperbolic splitting.

Then, using the compactness of the Grassmannian over the compact set $L$ this property persists,
by the normal hyperbolic theory \cite{HPS77}.

Indeed, we also have that the hyperbolic splitting $T_{q_n}M = N_n^s\oplus N_n^u$ with respect to linear Poincar\'e flow of $Y_n$ extends to a (maybe worse) hyperbolic splitting of a neighborhood of $L$. Thus, by uniform continuity of the respective quadratic forms $\J_n$, associated to the adapted metric of the hyperbolic splittings of $Y_n$, we obtain the continuous field of quadratic forms to $p$ on the set $L$.

By following \cite[Section 2.5]{ArSal2012}, given a strict $\J$-separated cocycle $B_t$ associated to the each flow $Y_n$ and its decomposition $T_{\Gamma}M$ and $\epsilon > 0$, we can find $\bar{\J}$ a smooth extension of $\J$ such that
\begin{enumerate}
\item $|\J_y(v)-\bar\J_y(v)|<\epsilon$ for all $v\in T_yM,
  y\in V$ ($C^0$-closeness on $V$);
\item
  $|\partial_t\J_{X_t(x)}(A_t(x)v)-\partial_t\bar\J_{X_t(x)}(A_t(x)v)|<\epsilon$
  for all $v\in T_{x}M, x\in L$ and $t\in\mathbb{R}$.
\end{enumerate}

Because monotonicity property is stronger than separation one, the above properties evidently hold to $\J$-monotonous cocycles.

Note that $\partial_t \J(P^t v)>0, \forall v \in T_xM, X(X_t(x))\neq 0, t \geq 0$, is equivalent to $\J(P^t v)>\J (v), \forall T_xN, X(x)\neq 0$ (\cite[Lemma 4.4]{ArSal2012}).

Then, we have $\J$-monotonicity on preperiodic orbits of $X$, by $C^1$-approximation.

Reciprocally, take a small neighborhood $U$ of $P_*(X)$ such that there is a $C^1$ neighborhood $\mathcal{V}$ of $X$ for which $P_*(Y) \subset U$ and suppose that such a field of quadratic forms is defined on. By Proposition \ref{prop:J-hyperbolic}, every singularity $\sigma \in U$ is hyperbolic.
The case of periodic orbits is analogous.

Shrinking $U$, if necessary, we may suppose that, for each periodic orbit and each singularity in $U$ of each $Y \in \mathcal{V}$, we have quadratic forms (still denoted $\J$) with the same features as before.
Indeed, since the quadratic form on each periodic orbit is $C^1$ along the flow, for any $Y \in \mathcal{V}$, shrinking $\mathcal{V}$ if necessary, we must have that any periodic orbit of $Y$ close enough to $X$ present stricly montonicity for the linear Poincar\'e flow, by normal hyperbolic theory. Then, it is done without change the signal of the derivative of the quadratic form.

Hence, for every vector field $Y$ in a sufficient small $C^1$-neighborhood $\mathcal{V}$ of $X$ and for every cocycle $B_t$ over $Y$ close to $A_t$, we have the same $\J$ properties. This is a consequence of the possibility of extension of the cone fields to an $\epsilon$-neighborhood of $L$ and the infinitesimal generator $D_{Y,B}$ of $B_t$ will be a linear map close to the infinitesimal generator $D$ of $A_t$.

Let $\mathcal{O}$ a preperiodic orbit of $X$ and $p \in \mathcal{O}$ a preperiodic point.

Again, we assume that the orbits of $(q_n, Y_n)$ tends, in Hausdorff sense, to a compact $X$-invariant set $L$, such that, we know, is contained into $P_*(X)$ and $\mathcal{O} \subset L$.

By hyphotesis, exists the $C^1$ field of quadratic forms $\J$ such that the linear Poincar\'e is strictly monotone on  $\mathcal{O}$. Then, we obtain a continuous $P^t$-invariant decomposition on the normal bundle $T_pN = N^-_p \oplus N^+_p$ such that $N^-_p$ contracts and $N^+_p$ expands uniformly, and it is well known that this is a robust property. Here, $N^-_p = \cap_{t>0} P^{-t}(p)(C_-(X_t(p)))$ and $N^+_p = \cap_{t>0} P^t(p)(C_+(X_{-t}(p)))$, for the corresponding negative/positive cones $C_{\pm}$ over the preperiodic orbit.

In other words, for every vector field $Y$, in a sufficient small $C^1$-neighborhood $\mathcal{V}$ of $X$, and for every cocycle $B_t$ over $Y$ close to $A_t$, and then to $P^t$, we must have the same $\J$-properties. This is a consequence of the cone fields extension to a neighborhood of $L$ and because the infinitesimal generator $D_{Y,B}$ of $B_t$ be a linear map close to the infinitesimal generator $D$ of the cocycle $A_t$ over $X$.

Moreover, by $C^1$-closeness between $X$ and $Y$, uniform continuity of the $\J$ derivative on preperiodic orbits in $L$ and robustness of strictly $\J$-monotonicity property, the signal of the derivative of the extension of the quadratic form clearly remains positive for nearby orbits of nearby vector fields.

The continuous extension of the cones is well known by classical hyperbolic theory and, more important to say, its properties remain valid for any vector field $C^1$-close enough to $X$.
As the periodic orbits of the approximating vector fields converge to a compact set of $X$ containing the preperiodic orbit, we obtain uniform continuity of the corresponding quadratic forms in a neighborhood of this compact set for nearby flows, see Remark \ref{rmk:forms-continuity} below.

If, for some $Y \in \mathcal{V}$, another periodic orbit is created, by $C^1$-closeness it is strictly $\J$-monotone for the linear Poincar\'e flow $P^t_Y$ associated to $Y$, since it comes from a preperiodic one of $X$ which is strictly $\J$-monotone for the linear Poincar\'e flow $P^t_X$, by hyphotesis. See \cite[Section 2.5.4]{ArSal2012} and the proof of \cite[Theorem 4.3]{ArSal2012}.

Moreover, cone criterion holds to $C^1$-close vector fields of $X$.

Hence, every periodic orbit for any $Y \in \mathcal{V}$ is hyperbolic. Therefore, $X$ is a star flow and the theorem is proved.$\diamond$

\begin{remark}\label{rmk:forms-continuity}
  A quadratic form $J: E_x \to \mathbb{R}$ be continuosly differentiable along the flow in $x \in U$ means that, for every section $Z$ of a finite dimensional vector bundle $E_U$ with base $U$, the map $x \mapsto J(Z(x))$ is continuous. In particular, in Theorem \ref{cor:star-flow} this is the case either for each preperiodic point $x$ of $X$, when we consider $X \in \mathfrak{X}^*(M)$, or for the hyperbolic periodic orbits of $Y_n$ and its hyperbolic continuation.
The $C^1$ assumption on $J$ along the flow of each $x \in U$ gives us that the map $t \mapsto J_{X_t(x)}(Z(X_t(x)))$ is continuously differentiable for each $C^1$ section $Z$ of $E_U$.
\end{remark}

Now, it is proved the second main result.

\proof [Proof of Theorem \ref{mthm:strong-homog-equiv}]

Since the linear Poincar\'e flow is strictly $\J$-monotone on each preperiodic orbit $\gamma \subset \Lambda^*$ implies that, if $\gamma$ is a closed orbit, then it is a hyperbolic subset of $\Lambda$, with a constant index, which we denote $\indi(\J)$. Moreover, taking a small enough neighborhood $U$ of $Per_*(X\vert_\Lambda)$ there exists some neighborhood $\mathcal{V}$ of $X$ such that $Per_*(Y\vert_{\Lambda_Y})\subset U$, $\forall Y \in \mathcal{V}$. If some periodic orbit $\gamma_Y$ is created by a small $C^1$ perturbation of $X$, it comes from a preperiodic orbit of $X$. Thus, $\gamma_Y$ is a hyperbolic closed orbit of $Y$, and must have index equal to $\indi(\J)$.

Hence, $\indi(\J)$ does not change by small differentiable perturbations of $X$ on a neighborhood of $\Lambda$, so the index of hyperbolic periodic orbits also does not change. Therefore, $\Lambda$ is strongly homogeneous for $X$.

Reciprocally, if $\Lambda$ is strongly homogeneous of index $\indi$, then cannot be there a non-hyperbolic periodic orbit. Otherwise, we can create two periodic orbits with different indices, by Frank's Lemma. Moreover, $X$ is a star flow in a neighborhood of $\Lambda$.
Hence, by Theorem \ref{cor:star-flow} we can define the desired field of quadratic forms $\J$, with fixed index $\indi(\J) = \indi$, defined on a neighborhood $U$ of $\Lambda$, where $U$ is the neighborhood for which $Per_*(Y\vert_{\Lambda_Y}) \subset U$ for any $Y$ close enough to $X$.$\diamond$

\proof [Proof of Corollary \ref{mcor:strong-homog}]

Note that Corollary \ref{mcor:strong-homog} follows from Theorem \ref{mthm:strong-homog-equiv}, since any periodic orbit is a preperiodic one and any singularity $\sigma$ is accumulated by regular orbits, it cannot present $\indi(\sigma) < \indi{\J}$. Indeed, $X \in \mathfrak{X}^{1}(M)$, $\J$ is a continuous field of quadratic forms and $X$ is $\J$- monotonic over any compact invariant nonsingular set $\Gamma$.$\diamond$

\subsection{Proof of Theorem \ref{mthm:strong-homog-to-quad}}\label{sec:thmD}

In this section, we prove our last main result.

To proceed the demonstration of Theorem \ref{mthm:strong-homog-to-quad}, we recall some definitions which are necessary from now.

Let $Z$ be a compact metric space and denote $\mathcal{M}(Z)$ the set of probabilities measures on the Borel $\sigma$-algebra of $Z$.
If $T: Z \to Z$ is a measurable map, we say that a probability measure $\mu$ is an invariant measure of $T$, if $\mu(T^{-1}(A)) = \mu(A)$, for every measurable set $A \subset Z$.  We say that $\mu$ is an invariant measure of $X$ if it is an invariant measure of $X_t$ for every $t \in \mathbb{R}$. We will denote by $\mathcal{M}_X$ the set of all invariant measures of $X$. A subset $Y\subset Z$ has \emph{total probability} if for every $\mu\in \mathcal{M}_X$ we have $\mu(Y)=1$ (see \cite{Man82}). The support of a measure $\mu$, denoted by $supp(\mu)$, is the set of points for which the measure is non-zero. An invariant measure is said to be \emph{atomic} if its support is either a closed orbit or a singularity.

A probability measure $\mu$ is an \emph{ergodic measure} if for every invariant set $A$ we have $\mu(A) = 1$ or $\mu(A) = 0$. Finally, a certain property is said to be valid in \emph{$\mu$-almost every point} if it is valid in the whole Z except, possibly, in a set of null measure.
\vspace{0.1in}

We recall the definition of $\delta$-closable points of \cite{Man82}.
We say that a point $x \in M \setminus Sing(X)$ is $\delta$-closable if, for any $C^1$ neighborhood $\SU \subset \mathfrak{X}^{1}(M)$ of $X$, there exists a vector field $Z \in \SU$, a point $z \in M$ and $T > 0$ such that:
 \begin{enumerate}
  \item $Z_T(z) = z$,
  \item $Z = X$ on $M \setminus B_{\delta} (X_{[0, T]}(x))$ and
  \item $dist(Z_t(z), X_t(x)) < \delta, \forall 0 \leq t \leq T$.
 \end{enumerate}
We denote by $\Sigma(X)$ the set of points of $M$ which are $\delta$-closable for any $\delta$ sufficiently small.

\proof [Proof of Theorem \ref{mthm:strong-homog-to-quad}]
If $\Lambda$ is a strongly homogeneous set for $X$ with singularities all of them hyperbolic, then $X$ is a star flow in $\Lambda$.

By Ergodic Closing Lemma, the $\delta$-closable set of $X$ has total probability.

If $x \in \Lambda$ is a regular $\delta$-closable point, then it is a pre-periodic point of index $\indi(\Lambda)$.

According the proof of \cite[Lemma 5.3]{GanLiWen2005}, we have a dominated splitting $E_x \oplus F_x$ of index $\indi(\Lambda)$ in $T_xM$, for all $x$ .

By Theorem \ref{thm:lyap-exp-sing-val}, we have
\begin{align*}
\chi^-_1 + \cdots + \chi^-_{k_1} \leq \sum_{i=1}^{k_1} \int \log r^-_i \ud\nu \ \textrm{and} \ \chi^+_1 + \cdots + \chi^+_{k_2} \geq \sum_{i=1}^{k_2} \int \log r^+_i \ud\nu,
\end{align*}
for any $k_1 \leq q, k_2 \leq p$.

Consider the continuous splitting 
\begin{align*}
   \beta^*(TM)\vert_{B_\sigma (\Lambda)} = N^s \oplus P \oplus N^u.
  \end{align*}
over $B(\Lambda)$, where $P$ is $\Phi_t$-invariant and $N^s, N^u$ are $\Psi^t$-invariant, which exists by \cite{Liao1979} and the star condition of $X$.

Also, following \cite[Lemma I.5]{Man88} and \cite{Liao1979}, there exists a $\hat{T} > 0$ and an ergodic $\Phi{\hat{T}}$-invariant probability measure $\mu$, with $\supp (\mu) \subset B_\sigma (\Lambda)$, satisfying
\begin{align*}
\int (\ln \Vert DX_{\hat{T}}\vert_{N^s_L}\Vert - \ln \Vert \Phi_{\hat{T}}\vert_{P_L}\Vert) \ud\mu(L) \geq 0.
\end{align*}
Then, we obtain that the ergodic probability measures are not atomic. Indeed, suppose that there is $\sigma \in \Lambda$ such that  and $\mu(B_\sigma (\Lambda))=1$. As $\Vert DX_{\hat{T}}\vert_{N^s_L}\Vert = \Vert \Phi_{\hat{T}}\vert_{E^{ss}_\sigma}\Vert$, and $E^{ss}_\sigma$ dominates $E^{cu}_\sigma$, we must have the next contradiction
\begin{align*}
\int (\ln \Vert DX_{\hat{T}}\vert_{N^s_L}\Vert - \ln \Vert \Phi_{\hat{T}}\vert_{P_L}\Vert) \ud\mu(L) < 0.
\end{align*}

Now, Birkhoff's ergodic theorem and Corollary \ref{cor:compos-max-exp} imply that the Lyapunov exponents on $E$ are negative and the sectional Lyapunov exponents are positive, in a total probability subset of $\Lambda$.

Moreover, for singularities $\sigma \in \sing(\Lambda)$ we have two possibilities:

First case: $\sigma$ is accumulated by recurrent orbits (including periodic orbits), then since $\indi(\sigma) \geq \indi(\Lambda)$, by Lemma \ref{le:GLW2} there is a dominated splitting $T_{\sigma}M = E_{\sigma} \oplus F_{\sigma}$, where $dim (E) = \indi(\Lambda)$.

Second case: Either there exists a dominated splitting on $T_{\sigma}M = E_{\sigma} \oplus F_{\sigma}$ with $dim (E) = \indi(\Lambda)$, which guarantees the definition of $\J$ such that $X$ is stricly $\J$-separated. Or, otherwise, since $\sigma$ is an isolated hyperbolic singularity with $\indi(\sigma) \geq \indi(\Lambda)$, we have an invariant splitting for which we only guarantee that $\J$ such that $X$ is (not strictly) $\J$-separated.

So, we have an invariant splitting $T_{\Lambda}M = E_{\Lambda} \oplus F_{\Lambda}$ which has uniformly angle bounded away from zero and $T_{\sigma}M = E_{\sigma} \oplus F_{\sigma}$ is dominated for every $\sigma \in \sing(X)$.

Now, \cite[Theorem C]{AraArbSal} implies that the corresponding decomposition $T_{\Lambda}M = E \oplus F$ is dominated of index $\indi(\Lambda)$.

By using the adapted metric for dominated splitting \cite{Goum07}, we obtain a field of $C^1$ non-degenerated quadratic forms $\J$ such that $X$ strictly $\J$-separated over $\Lambda$, as in \cite{ArSal2012}.

Now, to prove the $\J$-monotonicity, take a compact invariant set $\Gamma$ in $\Lambda^*$. Since $X$ is a star flow and $\Gamma$ is nonsingular, by \cite[Theorem A]{GanWen2006}, this set must be a hyperbolic one. So, by well known results, the linear Poincar\'e flow associated to $X$ is strictly $\J$-monotone on any compact invariant set $\Gamma \in \Lambda^*$.$\diamond$

\section*{Acknowledgments.}I dedicate this work to my son, \'Icaro Sol, star of my life. \\ I am also grateful to: the anonymous referee for fruitful comments and suggestions, IMPA - Instituto de Matemática Pura e Aplicada where the seminal version of this work has began in 2012, Universidade Federal do Rio de Janeiro and Universidade Federal da Bahia for the hospitality which helped to deeply improve this work.

\end{document}